\documentclass[envcountsame]{amsart}
\usepackage{amsmath, amsfonts, amssymb,amsthm}
\usepackage{graphics}
\usepackage{algpseudocode}
\usepackage[all]{xy}
\usepackage{url}

\usepackage{hyperref}

\usepackage{fancyhdr}

\newcommand{\cmfield}{CM-field}
\newcommand{\cmfields}{CM-fields}
\newcommand{\cmtype}{CM-type}

\def\F{\mathbf F}
\def\Q{\mathbf Q}
\def\Z{\mathbf Z}
\def\C{\mathbf C}
\def\dim{\mathrm{dim}}
\def\End{\mathrm{End}}
\def\O{\mathcal{O}}
\def\ord{\mathop{\mathrm{ord}}}

\newtheorem{theorem}{Theorem}
\newtheorem{lemma}[theorem]{Lemma}
\newtheorem{corollary}[theorem]{Corollary}
\theoremstyle{definition}
\newtheorem{algorithm}{Algorithm}{\bfseries}{\rmfamily}
\theoremstyle{remark}
\newtheorem{remark}[theorem]{Remark}
\newtheorem{example}[theorem]{Example}

\newcommand{\abs}[1]{\left|#1\right|}
\newcommand{\cl}{\mathrm{Cl}}
\newcommand{\liftC}{C}
\newcommand{\redC}{\widetilde{C}}
\newcommand{\liftA}{A}

\newcommand{\jacC}{J_C}
\renewcommand{\dim}{\mathop{\mathrm{dim}}}
\newcommand{\Gal}{\mathrm{Gal}}
\newcommand{\pfactor}{\mathfrak{p}_1\overline{\mathfrak{p}_1}\mathfrak{p}_2 }
\newcommand{\pifactor}{\mathfrak{p}_1^2\mathfrak{p}_2 }
\newcommand{\im}{\mathop{\mathrm{im}}}
\renewcommand{\index}[1]{}
\newcommand{\aimention}[1]{}
\newcommand{\mainentry}[1]{}
\newcommand{\toplinealgorithm}{

\noindent\underline{\hspace{\textwidth}}

}
\newcommand{\middlelinealgorithm}{

\noindent\underline{\hspace{\textwidth}}}
\newcommand{\bottomlinealgorithm}{
\vspace{-\baselineskip}

\noindent\underline{\hspace{\textwidth}}

}

\begin{document}

\protect \title{{A CM construction for curves of genus 2 with \MakeLowercase{${p}$}-rank~1}}

\protect \author[L.~Hitt]{Laura~Hitt~O'Connor}
\address{Laura~Hitt~O'Connor and Gary~McGuire
\\
School of Mathematical Sciences\\
University College Dublin, Ireland}
\email{\href{mailto:hitt36@gmail.com}{hitt36@gmail.com}, \href{mailto:gary.mcguire@ucd.ie}{gary.mcguire@ucd.ie}}
\thanks{Research of the first author supported by Science Foundation Ireland Post-Doctoral Grant
07/RFP/ENM123}
\author[G.~McGuire]{Gary~McGuire}
\thanks{Research of the second author supported by the Claude Shannon Institute, Science Foundation Ireland Grant 06/MI/006}
\author[M.~Naehrig]{Michael~Naehrig}
\address{Michael~Naehrig\\
Department of Mathematics and Computer Science, 
Eindhoven University of Technology, Den Dolech 2, 5600 MB Eindhoven, the Netherlands\\
and Microsoft Research, One Mirosoft Way, Redmond, WA 98052, USA}
\email{\href{mailto:michael@cryptojedi.org}{michael@cryptojedi.org}}
\urladdr{\url{http://www.cryptojedi.org/users/michael/}}
\author[M.~Streng]{Marco~Streng}
\address{Marco~Streng\\ Mathematisch Instituut\\
Universiteit Leiden\\
Postbus 9512\\
2300 RA Leiden\\
the Netherlands}
\email{\href{mailto:streng@math.leidenuniv.nl}{streng@math.leidenuniv.nl}}
\urladdr{\url{http://www.math.leidenuniv.nl/~streng/}}

\pagestyle{fancy}
    \fancyhead{}
    \fancyfoot{}
\fancyhead[CE]{L.~Hitt, G.~McGuire, M.~Naehrig, and M.~Streng}
\fancyhead[CO]{Curves of genus 2 with $p$-rank~1, preprint of \today}
\fancyfoot[CO,CE]{\thepage }
\setlength{\headheight}{14pt}

\begin{abstract}
We construct Weil numbers corresponding
to genus-$2$ curves with $p$-rank~$1$
over the finite field $\F_{p^2}$ of $p^2$ elements.
The corresponding curves can be constructed using
explicit CM constructions.
In one of our algorithms,
the group of $\F_{p^2}$-valued points of the Jacobian has prime order,
while another allows for a 
prescribed embedding degree with respect to a subgroup
of prescribed order.
The curves are defined over $\F_{p^2}$ out of necessity:
we show that
curves of $p$-rank $1$ over $\F_p$ for large~$p$ cannot
be efficiently constructed using explicit CM constructions.
\end{abstract}

\protect \maketitle

\section{Introduction}\label{sec:intro}

The \emph{$p$-rank}\mainentry{p-rank@$p$-rank}
of an abelian variety $A$ over a field
$k$ of characteristic $p$ is the integer $r=r(A)$ such that
the group $A[p](\overline{k})$ of $p$-torsion points
over an algebraic closure $\overline{k}$ of $k$ has order~$p^r$.
It satisfies $0\leq r\leq g$, where $g$ is the dimension
of $A$, and we call $A$ \emph{ordinary}\index{ordinary} if $r$ is equal to~$g$.
If $A$ is \emph{supersingular},\index{supersingular} that is, if $A$ becomes isogenous
over $\overline{k}$ to a product of supersingular elliptic curves,
then we have $r=0$, and
the converse holds for abelian surfaces:
if $r=0$ and $g=2$, then $A$ is supersingular.

This shows that for an abelian surface $A$, besides the
ordinary and supersingular cases,
there is only one \emph{intermediate}\index{intermediate abelian variety} case:
the case where $A$ has $p$-rank~$1$.
Most CM constructions of curves of genus two
\cite{spallek,weng,freeman-g2,freeman-stevenhagen-streng}
generate curves that are
ordinary with probability tending to~$1$,
while another \cite{rubin-silverberg}
constructs only supersingular curves.
We focus on the intermediate case,
for which no constructions existed yet.

The $p$-rank $r(A)$ depends only on the \emph{isogeny class}
of $A$ over~$\overline{k}$, and any simple abelian surface $A$
of $p$-rank~$1$ over a finite field $k$ is
isogenous to the Jacobian of a curve
over $k$ of genus~$2$
(see Section~\ref{sec:prank1numbers}).
By the $p$-rank of a curve $C$,\index{p-rank@$p$-rank}
we mean the $p$-rank of its Jacobian~$\jacC$.

Let $k$ be the finite field of order $q=p^n$.
The Frobenius endomorphism $\pi$ of a simple abelian variety over $k$
is a \emph{Weil $q$-number},\index{Weil number} i.e.,
an algebraic integer $\pi$ such that
$\abs{\pi}^2=q$ holds for every embedding of
the field $K=\Q(\pi)$ into the complex numbers.
A theorem of Honda and Tate \cite{honda-tate} states that
this defines a bijection between the set of isogeny classes of
simple abelian varieties over $k$
and the set of Weil $q$-numbers
up to Galois conjugacy.

We characterize those Weil numbers corresponding
to abelian surfaces with $p$-rank $1$ in
Section~\ref{sec:prank1numbers},
show their existence in Section~\ref{sec:existence}
and give algorithms for finding them in
Section~\ref{sec:algorithms}.
In Section~\ref{sec:existence} we also explain why
curves of $p$-rank $1$ over $\F_p$ for large~$p$ cannot
be efficiently constructed using explicit CM constructions.

The construction of an abelian variety $A$ corresponding to
a given Weil $q$-number $\pi$ dates back to Shimura
and Taniyama \cite{shimura-taniyama} and Honda \cite{honda}.
It exhibits $A$ as the reduction of a
characteristic-$0$ abelian variety with
\emph{complex multiplication (CM)}\index{complex multiplication}\index{CM} by $\Z[\pi]$
and is also known as the \emph{CM method}.\index{CM method}
We explain this explicit CM construction
in Section \ref{sec:cmmethod}.
For now, it suffices to say that the computational
complexity of this construction
grows very rapidly with the size of the field~$K=\Q(\pi)$.
Therefore, our algorithms will
look for Weil $q$-numbers $\pi$
only in fixed small input fields~$K$.

Let $A$ be an abelian variety over the finite field $k$
and suppose that $A(k)$ has a subgroup of prime order $r$.
The \emph{embedding degree}\index{embedding degree} of $A$ with respect to $r$ is the degree
of the field extension $k(\zeta_r)/k$, where $\zeta_r$ is a primitive
$r$-th root of unity.
The \emph{Weil} and \emph{Tate} pairings\index{pairing} on $A$
with respect to $r$ have
their image in $\langle \zeta_r\rangle\subset k(\zeta_r)^*$,
and in order
to compute these pairings,
one needs to work with $k(\zeta_r)$.
As the embedding degree is the order of $q$ in $(\Z/r\Z)^*$,
it is close to $r$ for most curves,
while for \emph{pairing-based cryptography}, one wants
$r$ to be large and
the embedding degree to be small.
Algorithm \ref{algo3} in Section~\ref{sec:algorithms}
provides curves with $p$-rank $1$
and a prescribed
small embedding degree.

We used our algorithms to compute various examples,
which we give in Section~\ref{examples}.
Each example was computed in a few seconds on a standard PC.

\section{Characterization of abelian surfaces of
\texorpdfstring{$p$}{p}-rank 1}\label{sec:prank1numbers}

It follows from the definition that the $p$-rank
$r(A)$ of an abelian variety $A$
does not change under extensions of the base field,
and that it satisfies $r(A\times B)=r(A)+r(B)$ for any
pair of abelian varieties $A$ and $B$.
It is also well-known that the $p$-rank is invariant
under isogeny (see Lemma \ref{lem:prank1formula} below).
In particular, the non-simple abelian surfaces of $p$-rank~$1$
are exactly those isogenous to the product of an ordinary
and a supersingular elliptic curve. Both
types of elliptic curves are well understood, so
we focus on \emph{simple}\index{simple abelian variety} abelian surfaces.
We use the word \emph{isogeny} to mean isogeny defined over the base
field $k$, unless otherwise stated. We use the same convention
for the definition of \emph{simple} abelian variety.

Our algorithms are based on a characterization
of Weil numbers corresponding to simple abelian surfaces
of $p$-rank~$1$, which we give in this section.
A major part of this characterization
can already be found in Goren~\cite{goren} and
Gonzalez~\cite[proof of Thm.~3.7]{gonzalez},
but we give a proof,
as this result is the foundation of our construction.

Let $k$ be the finite field of $q=p^n$ elements
and let $\pi$ be a Weil $q$-number.
For every embedding of the field $K=\Q(\pi)$ into~$\C$,
complex conjugation on $K$ is given by $\pi\mapsto q/\pi$.
As this automorphism of $K$ doesn't depend
on the choice of the embedding, we denote it
by~$x\mapsto \overline{x}$ and call it complex conjugation.
If we let $K_0$ be the fixed field of
complex conjugation, then
$K_0$ is totally real and $K$ is either equal to $K_0$
or it is a \emph{\cmfield},\index{CM-field} that is, a totally imaginary
quadratic extension of a totally real number field.
\begin{lemma}\label{lem:mainlemma}
  A simple abelian variety $A$ over the field $k$
  of $q=p^n$ elements has dimension $2$ and
  $p$-rank~$1$ if and only if the following
  three conditions hold for its Frobenius endomorphism~$\pi$:
\begin{enumerate}
 \item the field $K=\Q(\pi)$ is a {\cmfield} of degree $4$,
 \item the prime $p$ factors in $K$ as $p\O_K=\pfactor ^e$,
       with $e\in\{1,2\}$, and
 \item we have $\pi\O_K=\mathfrak{p}_1^n\mathfrak{p}_2^{\smash{en/2}}$ with $e$ as in~\textup{(2)}.
\end{enumerate}
Note that condition \textup{(3)} implies that $en$ is even.
\end{lemma}

We prove Lemma \ref{lem:mainlemma} using
the following formula for the $p$-rank of an abelian variety.
\begin{lemma}[{\cite[Prop.~3.1]{gonzalez}}]\label{lem:prank1formula}
  Let $A$ be a simple abelian variety over $k$ and let $K=\Q(\pi)$,
  where $\pi$ is the Frobenius endomorphism of $A$.
  There is an integer $m$ such that $2\dim(A)=m\deg K$ holds.
  Suppose that $p$ factors in $K$ as
  $p\O_K=\prod_i \mathfrak{p}_i^{e_i}$ and let
  $f_i$ be given by $\#(\O_K/\mathfrak{p}_i)=p^{f_i}$.
  Then we have
  $r(A)=\sum me_if_i$, where the sum is taken over those $i$
  for which $\pi\not\in \mathfrak{p}_i$ holds.
\end{lemma}
\begin{proof}
The \emph{degree} $\deg g$ and \emph{separable degree} $\deg_{\mathrm{s}} g$
of an isogeny $g:A\rightarrow B$
of abelian varieties are defined to be the degree and separable
degree of the induced embedding of function fields
$g^*:k(B)\rightarrow k(A)$.
We have $\#(\ker g)(\overline{k})=\deg_{\mathrm{s}}g$, hence
$p^{r(A)}$ is the separable degree of the multiplication-by-$p$ map on~$A$.
As the separable degree is multiplicative under composition,
we find that the $p$-rank of $A$ depends only on its isogeny class,
hence we can assume
that $\End_{k} A$ contains the maximal order $\O_K$ by
\cite[Prop.~7 in \S7.1]{shimura-taniyama}.

The existence of $m$ follows from \cite[Thm.~1(2)]{honda-tate}.
The theory in \cite[\S7]{shimura-taniyama} shows how to factor the multiplication-by-$p$
map into multiplication-by-$\mathfrak{p}_i$ maps for prime ideals $\mathfrak{p}_i$,
and that the multiplication-by-$\mathfrak{p}_i$ map has degree $p^{f_im}$.
The Frobenius endomorphism $\pi$
is totally inseparable by
\cite[Thm.~1(iii) in \S2.8]{shimura-taniyama},
hence so is multiplication-by-$\mathfrak{p}_i$ if $\mathfrak{p}_i$ contains~$\pi$.
If $\mathfrak{p}_i$ is coprime to $\pi$, then
\cite[Prop.~6 in \S2.8]{shimura-taniyama} shows that it is separable,
hence satisfies $\deg_{\mathrm{s}}\mathfrak{p}_i=\deg\mathfrak{p}_i$.
\end{proof}
\begin{proof}[\textbf{\textup{Proof of Lemma \ref{lem:mainlemma}}}]
  If $A$ has dimension $2$ and $p$-rank $1$, then Lemma
  \ref{lem:prank1formula} tells us $m=1$, hence $K$
  has degree $4$ and
  exactly one prime $\overline{\mathfrak{p}_1}|p$
  with $\pi\not\in \overline{\mathfrak{p}_1}$, which
  is unramified and has residue degree~$1$.
  This implies $p\mathcal{O}_K=\mathfrak{p}_1\overline{\mathfrak{p}_1}\mathfrak{q}$,
  where $\mathfrak{q}$ is prime in the fixed field $K_0$
  of complex conjugation.

  To prove that (2) and (3) hold, it now suffices
  to prove that $\mathfrak{q}$ does not split in $K/K_0$.
  Suppose that it does, say $\mathfrak{q}=\mathfrak{q}_1\overline{\mathfrak{q}_1}$.
  Then by \cite[Thm.~1(1)]{honda-tate}, the fact $m=1$
  implies that $\ord_{\mathfrak{q}_1}(\pi)$ is either $0$
  or equal to the degree $n=\deg k/\F_p$.
  We also have $\ord_{\mathfrak{q}_1}(\pi)+\ord_{\overline{\mathfrak{q}_1}}(\pi)=\ord_{\mathfrak{q}_1}(\pi\overline{\pi})=n$,
  hence one of $\mathfrak{q}_1$ and $\overline{\mathfrak{q}_1}$ does not divide $\pi$, i.e., contradicts
  uniqueness of~$\overline{\mathfrak{p}_1}$.
  
  Conversely, if $\pi$ satisfies (1), (2), and (3), then
  Lemma \ref{lem:prank1formula} implies $r(A)=m$
  with $2\dim (A)=m\deg K$ and
  \cite[Thm.~1(1)]{honda-tate} implies $m=1$.
\end{proof}
\begin{corollary}
   A simple abelian surface $A/k$ of $p$-rank~$1$ is absolutely simple,
   that is, simple over $\overline{k}$,
   and is isogenous to the Jacobian of a curve $C$ over~$k$.
\end{corollary}
\begin{proof}
   Suppose that $k'/k$ is an extension of degree $d$ such that
   we have $A_{k'}\sim E\times F$.
   The Frobenius endomorphism of $A_{k'}$ is $\pi^d$ and the
   characteristic polynomial of its action on the $\ell$-adic
   Tate module of $A$ for $l\not=p$ is the product of the (quadratic)
   characteristic polynomials of the action on the Tate modules
   of $E$ and $F$.
   
   On the other hand, part (3) of Lemma \ref{lem:mainlemma}
   implies that $\Q(\pi^d)$ is equal to $K$, which is a field
   of degree~$4$. This is a contradiction, hence $A$ is absolutely simple.

   By \cite[Theorem 4.3]{maisner-nart},
   any absolutely simple abelian surface over a finite field $k$
   is isogenous to the Jacobian of a curve.
\end{proof}
\begin{remark}
  The conditions \textup{(1)}, \textup{(2)}, and \textup{(3)} of Lemma \ref{lem:mainlemma}
  are equivalent to conditions (M)
  of Theorem 2.9 of Maisner and Nart~\cite{maisner-nart},
  i.e., to the characteristic polynomial $f=X^4-a_1X^3+(a_2+2q)X^2-qa_1X+q^2$
  of $\pi$ satisfying
  \begin{enumerate}
  \item $f$ is irreducible,
  \item $\ord_p(a_1)=0$,
  \item $\ord_p(a_2)\geq n/2$,
  \item and that $(a_2+4q)^2-4qa_1^2$ is not a square
      in the ring of $p$-adic
  integers~$\Z_p$.
  \end{enumerate}
\end{remark}
\begin{remark}
  For an elliptic curve $E$ over a finite field~$k$,
  the rank of the $\Z$-algebra $\End_{\overline{k}}(E)$ of
  $\overline{k}$-endomorphisms is either $2$ or $4$,
  and these cases correspond exactly to the cases $r(E)=1$ and $r(E)=0$.
  
  For abelian surfaces $A$, the $p$-rank $r(A)$ cannot be computed
  from the $\Z$-rank of the endomorphism algebra.
  In fact, for absolutely simple abelian surfaces $A$,
  the ring
  $\mathrm{End}_{\overline{k}}(A)\otimes\Q$
  is always a {\cmfield} of degree $4$,
  while both $r(A)=1$ and $r(A)=2$ occur
  (see also \cite[Thm 3.7(ii)]{gonzalez}).
\end{remark}

\section{Existence of suitable Weil numbers}\label{sec:existence}

Let $p$ be a prime that factors in $K$
as in (2) of Lemma \ref{lem:mainlemma}.
The fact that not all primes over $p$ have the same ramification index
or residue degree implies that the degree-$4$
extension $K/\Q$ is not Galois.
As $K$ has a non-trivial automorphism, complex conjugation, the
normal closure
$L$ of $K$ has Galois group~$D_4$.
We therefore have to restrict to non-Galois quartic number fields $K$
with Galois group~$D_4$.

In the case $e=2$, the prime $p$ ramifies in~$K$, hence divides
its discriminant. Since
explicit CM constructions are feasible only for small fields~$K$,
i.e., fields $K$ of small discriminant,
this means that we can construct the curve $C$ corresponding to $\pi$
only for very small values of~$p$.
For such small values of $p$, not only are the curves
less interesting, especially from a cryptographic point of view,
it also becomes possible to construct them using a more
direct approach such as by enumerating all curves $C$ of genus~$2$
over $\mathbf{F}_p$ and computing the group orders
of their Jacobians.
Therefore, we will focus on the case~$e=1$.
For $e=1$, condition (3) of Lemma \ref{lem:mainlemma} implies~$2|n$,
so that curves are defined only
over fields containing~$\mathbf{F}_{p^2}$. This is the reason
why we construct our curves over $\mathbf{F}_{p^2}$ and not
over~$\F_{p}$, and 
this is why
curves of $p$-rank $1$ over $\F_p$ for large~$p$ cannot
be efficiently constructed using explicit CM constructions.

We have found that all fields with $p$-rank-$1$
Weil $p^2$-numbers are quartic non-Galois \cmfields.
However, not all quartic non-Galois
{\cmfields } have $p$-rank-$1$ Weil $p^2$-numbers,
and we give a complete characterization in
Section~\ref{sec:suffnec}.

For now, we give two lemmas that
put a condition on the {\cmfields } $K$
that is slightly too strong, but is easy
to check and is satisfied by `most' non-Galois quartic \cmfields.
\begin{lemma}\label{lem:alpha}
  Let $K$ be a quartic {\cmfield } and let $p$ be a prime
  that factors in $K$ as
  $p\O_K=\pfactor $.
  Suppose that $\mathfrak{p}_1=\alpha\O_K$ is principal. Then
  $\pi=\alpha\overline{\alpha}^{-1}p$ is a Weil $p^2$-number
  that satisfies the conditions of Lemma \ref{lem:mainlemma}.
\end{lemma}
\begin{proof}
  The number $\pi$ satisfies $\pi\overline{\pi}=p^2$,
  hence is a Weil $p^2$-number.
  Conditions (1) and (2) of Lemma \ref{lem:mainlemma}
  are satisfied by assumption.
  Moreover, we have $\mathfrak{p}_2
  =p(\mathfrak{p}_1\overline{\mathfrak{p}_1})^{-1}
  =p(\alpha\overline{\alpha})^{-1}\O_K$,
  so that we have $\pi\O_K=\pifactor $,
  i.e., condition (3) is also satisfied.
  \end{proof}
  The condition on $p$ of Lemma \ref{lem:alpha}
  is stronger than the condition
  that there exists a Weil $p^2$-number in $K$ with $e=1$.
  The following lemma gives a necessary and sufficient criterion on $K$
  for the existence of primes $p$ satisfying this stronger condition.

  For a non-Galois quartic {\cmfield } $K$,
  let $L$ be its normal closure over $\Q$
  and let $d$ be the discriminant of the real quadratic subfield
  $K_0$ of~$K$.
  Then we have $K=K_0(\sqrt{r})$ for a totally negative element $r\in K_0$,
  and $s=N_{K_0/\Q}(r)\in\Q$ is not a square, because
  $K$ is non-Galois.
  Let $d^{\mathrm{r}}$ be the discriminant
  of the real quadratic field $K_0^{\mathrm{r}}=\Q(\sqrt{s})$.
  Note that this field is independent of the choice of~$r$.
  Indeed, the element $r$ is well-defined up to squares in $K_0^*$, 
  hence $s$ is well-defined up to squares in~$\Q^*$.

  A \emph{prime discriminant}\index{prime discriminant} is a number
  that is $-4$ or $\pm 8$ or is $\pm p\equiv 1\pmod 4$ for
  an odd prime~$p$.
  The discriminant of a quadratic field can be written
  uniquely as a product
  of distinct prime discriminants
  in which at most one even factor occurs.
  \begin{lemma}\label{lem:alphadensity}
  Let $K$ be a non-Galois quartic \cmfield.
  The following are equivalent
  \begin{enumerate}
  \item there exists a prime $p$ that factors in $K$
  as $p\O_K=\pfactor $
  with $\mathfrak{p}_1$ principal;
  \item the Dirichlet density of
  the set of primes $p$ as in \textup{(1)}
  is $(4 h_K)^{-1}$,
  where $h_K$ is the class number of~$K$;
  \item there is a prime that ramifies in $L/K$;
  \item not all prime discriminants in the discriminant
        factorization of $d^{\mathrm{r}}$
        occur in that of $d$.
  \end{enumerate}
  \end{lemma}
  \begin{proof}
  The implication $(2)\Rightarrow(1)$ is trivial.
  Now suppose that (1) holds, so the decomposition
  group of $\mathfrak{p}_1$ in $\Gal(L/\Q)$ is $\Gal(L/K)$
  and the ideal class of $\mathfrak{p}_1$ is trivial.
  By the Artin isomorphism $\cl_K\rightarrow \Gal(H/K)$,
  this implies that the
  decomposition group of $\mathfrak{p}_1$ in $\Gal(H/K)$
  is trivial for the Hilbert class field $H$ of~$K$.
  As the decomposition group of $\mathfrak{p}_1$ in
  $\Gal(L/K)$ is non-trivial, this implies that $L$ is not contained
  in the maximal unramified abelian extension $H$ of~$K$,
  so $L/K$ ramifies at some prime and (3) holds.
  
  For the proof of $(3)\Rightarrow(2)$, we use again that the primes
  $p$ as in (1) are those
  for which there exists a prime in $L$ over $p$
  with decomposition group $\Gal(L/K)$ in $L/\Q$ and
  trivial decomposition group $H/K$.
  Let $M\supset H$ be Galois over $\Q$.
  Since (3) implies $L\cap H=K$, we find
  $\Gal(HL/K)=\Gal(H/K)\times \Gal(L/K)$ and hence
  that exactly $1$ in every
  $8h_K$ elements $\sigma\in \Gal(M/\Q)$ satisfies
  $\langle \sigma_{|L}\rangle=\Gal(L/K)$ and $\sigma_{|H}=1$.
  The conjugation class of $\Gal(L/K)$ in $\Gal(L/\Q)$ has two elements,
  hence the set of all $\sigma$ yielding the appropriate
  factorization is twice as large, i.e.,
  consists of $1$ in every $4h_K$ elements of $\Gal(M/\Q)$.
  By Chebotarev's density theorem~\cite[Theorem 13.4]{neukirch}, this
  implies that the density of primes with this factorization
  is $(4h_K)^{-1}$, which proves~(2).
  
  Now, it remains to prove $(3)\Leftrightarrow(4)$.  
  Let $L_0$ be the compositum of $K_0$ and $K_0^{\mathrm{r}}$ in~$L$.
  A prime $q\in\Z$ ramifies in $L/K$ if and only if its inertia group
  in $\Gal(L/\Q)$ contains $\Gal(L/K)$ or its conjugate.
  This is equivalent to $q$ ramifying in $L_0/K_0$,
  that is, to the prime discriminant in $d^{\mathrm{r}}$ corresponding
  to $q$ not occurring in the prime discriminant factorization of~$d$.
  \end{proof}
\begin{example}\label{ex:firstexample}
  The field $K=\Q[X]/(X^4+12X^2+2)$ does not satisfy
  the conditions of Lemma~\ref{lem:alphadensity}, because
  it has $d=8\cdot 17$ and $d^{\mathrm{r}}=8$.
\end{example}
For `most' non-Galois quartic \cmfields~$K$, the discriminant $d^{\mathrm{r}}$ does not
divide~$d$, in which case the conditions of Lemma~\ref{lem:alphadensity}
hold.
This means that if we try to find our Weil numbers
by taking random primes $p$ and checking if there exists
a Weil $p^2$-number $\pi\in K$ as in Lemma~\ref{lem:mainlemma},
then we have a probability $(4h_K)^{-1}$ of success.

\section{The algorithms}\label{sec:algorithms}

The discussion in Section~\ref{sec:existence} leads to the following
algorithm.

\toplinealgorithm

\begin{algorithm}
\label{algo1}
\hfill \\
\textbf{Input:} A non-Galois {\cmfield } $K$ of degree $4$ and a positive integer~$\ell$.\\
\textbf{Output:} A prime $p$ of $\ell$ bits and a Weil $p^2$-number $\pi$ corresponding
  to the Jacobian $\jacC$ of a curve of genus~$2$ over $\F_{p^2}$ such that
  $\# \jacC(\F_{p^2})$ is prime.
\middlelinealgorithm
\begin{enumerate}
\item\label{algorithm1step1} Take a random positive integer
       $p$ of $\ell$ bits.
\item If $p$ is prime, continue. Otherwise, go to Step~\ref{algorithm1step1}.
\item If $p\mathcal{O}_K$ factors as $\pfactor $,
       continue.
       Otherwise, go to Step~\ref{algorithm1step1}.
\item If $\pifactor $ is principal, let $\pi_0$ be a generator
and let $v=\pi_0\overline{\pi_0}p^{-2}\in\mathcal{O}_{K_0}^*$.
Otherwise, go to Step~\ref{algorithm1step1}.
\item If we have $v=N_{K/K_0}(w)$ for some $w\in\O_K^*$,
then put $\pi=w^{-1}\pi_0$.
      Otherwise, go to Step~\ref{algorithm1step1}.
\item  If $N(u\pi-1)$ is prime for some $u\in \{\pm 1\}$,
       then replace $\pi$ by~$u\pi$.
       Otherwise, go to Step~\ref{algorithm1step1}.
\item \Return $p,\pi$.
\end{enumerate}
\bottomlinealgorithm
%

\end{algorithm}

Note that the group order $N(\pi-1)$
of $\jacC$ has about $4\ell$ bits
since we have $N(\pi-1)\approx N(\pi)=p^4$.

\begin{theorem}\label{thm:algorithm1}
  If Algorithm~\ref{algo1} terminates, then the output is correct.
  
  Fix the input field $K$ and assume that it satisfies
  the conditions
   of Lemma~\ref{lem:alphadensity}. If $K$ has no prime ideal
   of norm $2$, and no prime above $2$ is ramified in $K/K_0$,
   then
   the heuristic expected runtime of the
  algorithm is polynomial in $\ell$.
\end{theorem}
\begin{proof}
  The output $\pi$ is a Weil $p^2$-number
  satisfying the conditions of Lemma~\ref{lem:mainlemma},
  and the corresponding abelian surface $A$ has $\#A(\F_{p^2})=N(\pi-1)$
  rational points, which
  proves that the output is correct.

  All numbers encountered have logarithmic absolute values and heights
  that are bounded linearly in $\ell$, while the field $K$ is fixed.
  This shows that, using the algorithms of \cite{cohen}, 
  all steps, including the primality and principality tests,
  as well as finding a generator of $\mathfrak{p}_1^2\mathfrak{p}_2$
  and trying to extract a square root of $v$, take time
  polynomial in~$\ell$.
  It therefore suffices to prove that the
  heuristic expected number of iterations
  of Step~\ref{algorithm1step1} is quadratic in~$\ell$.

  The number $p$
  has a heuristic probability $1/(\ell\log 2)$ to be prime
  by the Prime Number Theorem.\index{Prime Number Theorem}
  This shows that for each time Step~3 is reached,
  one expects to run Step~1 about $\ell\log 2$ times.

  We will `prove' that the heuristic bound holds
  even if we restrict in Step~3 to $\mathfrak{p}_1$
  principal and generated by $\alpha$.
  By Lemma~\ref{lem:alphadensity},
  the density of the set of primes $p$
  that factor in the appropriate way and
  for which $\alpha$ exists is $(4h_K)^{-1}$,
  so we arrive at Step~4 (with $\mathfrak{p}_1=(\alpha)$)
  with probability~$(4h_K)^{-1}$.
  
  Note that $\pi=-\alpha\overline{\alpha}^{-1}p$
  is a generator of $\mathfrak{p}_1^2\mathfrak{p}_2$,
  so we pass Step~4 with $\pi_0=w\pi$ for some unit
  $w\in\O_K^*$.
  
  Note that we have $p^2=\pi\overline{\pi}$,
  hence $v=w\overline{w}$, proving that we pass Step~5 as well.
  
  We now only need to show that $N(\pi-1)$
  is prime with sufficiently high probability.
  Treating $\alpha$ as a random element of~$\O=\O_K$,
  we wish to know the probability that $X=N(\pi-1)$
  is prime, i.e., not divisible by any prime $q<X$.
  For each such $q$, we consider the homomorphism
  $$\varphi:(\O/q\O)^*\rightarrow (\O/q\O)^*:
    x \mapsto x\overline{x}^{-1}N(x),$$
  which sends $(\alpha\bmod q)$ to $(-\pi\bmod q)$.
  Now we have $q|N(\pi-1)$ if and only if
  $\pi\equiv 1\pmod{\mathfrak{q}}$
  for some prime $\mathfrak{q}|q$ of~$K$.
  Let $\varphi_{\mathfrak{q}}$ be the composition
  of $\varphi$ with the natural map
  $(\mathcal{O}/q\mathcal{O})^*\rightarrow (\mathcal{O}/\mathfrak{q})^*$.
  Note that we have $\pi\equiv 1\pmod{\mathfrak{q}}$
  if and only if $\alpha$ is an element of
  $\varphi_{\mathfrak{q}}^{-1}(-1)$.
  If we define
  $$P_q=1-\frac{\#\bigcup_{\mathfrak{q}|q} \varphi_{\mathfrak{q}}^{-1}(-1)}
    {\#(\O/q\O)^*},$$
  then the heuristic probability of $q\nmid N(\pi-1)$ equals~$P_q$.
  As the homomorphism $\varphi$ sends $1$ to~$1$,
  we find $P_q>0$ for all $q>2$.

  For $q=2$, note that we have $N(x)=1$. Then for all $\mathfrak{q}\mid q$
  with $\overline{\mathfrak{q}}=\mathfrak{q}$, take
  $(x\bmod\mathfrak{q})\in (\mathcal{O}/\mathfrak{q})^*$ with $x\not=\overline{x}$,
  which is possible, because $2$ is unramified in $K/K_0$.
  For $\mathfrak{q}\mid q$ with $\overline{\mathfrak{q}}\not=\mathfrak{q}$,
  take exactly one of $(x\bmod\mathfrak{q})$ and $(x\bmod\overline{\mathfrak{q}})$
  equal to $1$,
  which is possible because $\mathfrak{q}$ has norm $\geq 4$.
  Then $x\overline{x}^{-1}\not\equiv 1\equiv -1\pmod{\mathfrak{q}}$ for all $\mathfrak{q}\mid q$,
  which proves $P_2>0$.

We use the lower bound $P_q>0$
  for $q\leq 17$.

  For $q\geq 19$, note that we have
\[P_q \geq 1-\sum_{\mathfrak{q}\mid q}\frac{\# \ker \varphi_{\mathfrak{q}}}{\#(\O/q\O)^*}
    \geq 1-\sum_{\mathfrak{q}\mid q}\frac{1}{\#\im \varphi_{\mathfrak{q}}}\]
  and that
$\smash{ \im \varphi_{\mathfrak{q}}\supset \varphi_{\mathfrak{q}}(\mathbf{F}_q^*)=(\mathbf{F}_q^*)^4}$
  has order $\geq (q-1)/4$, hence we have
  \[ P_q \geq 1-4\frac{4}{q-1}>1-\frac{17}{q}.\]

  We thus find heuristically that $N(\pi-1)$ is prime with probability
  at least a positive constant times
  $$Y=\prod_{\substack{19\leq q<X\\ \mathrm{prime}}} \left(1-\frac{17}{q}\right).$$
  We find $\log(Y)>-\sum_{q}\frac{17}{q}$,
  and the right hand side, by Mertens' theorem
  \cite[Thm.~427 in~22.7]{hw},
  is $17\log\log X$ plus something that
  converges to a constant if $X$ tends to infinity.
  In particular, we find that
  $1/Y$ is at most polynomial in $\log X\approx 4\ell$, which is what
  we needed to prove.  
\end{proof}
\begin{remark}
For more detailed heuristics on prime order Jacobians
of curves of genus~$2$
than what is in the proof of Theorem~\ref{thm:algorithm1},
see \cite[\S5.2.2]{wengthesis}.
\end{remark}
\begin{remark}
The conditions of Lemma~\ref{lem:alphadensity} are
sufficient in Theorem~\ref{thm:algorithm1} and,
as we said before, they hold for `most' non-Galois quartic \cmfields.
They are however not necessary, and
we give strictly weaker conditions
in Section~\ref{sec:suffnec}.
\end{remark}
The following lemma shows that the 
conditions on the decomposition of $2$ in $K$
are necessary in Theorem~\ref{thm:algorithm1},
and that these conditions
are not specific to $p$-rank~$1$,
or even to abelian surfaces.
These conditions vanish however if one allows
the group order to be `almost prime'\index{almost prime} in
the sense that it is a prime times
a `small' (say $\leq 16$) positive integer.
\begin{lemma}
  Let $\pi$ be the Frobenius endomorphism of an abelian variety $A$
  over a finite field $k$ of odd characteristic,
  and let $K=\Q(\pi)$.
  If one of the following conditions holds, then
  the order of $A(k)$ is even.
  \begin{enumerate}
  \item $K$ has a prime ideal $\mathfrak{q}$ of norm $2$,
  \item $K$ is totally real, or
  \item $K$ is a {\cmfield } with totally real subfield $K_0$
  and $K$ has a prime ideal $\mathfrak{q}|2$
  that is ramified in~$K/K_0$.
  \end{enumerate}
\end{lemma}
\begin{proof}If $\mathfrak{q}$ has norm $2$,
then we have $\pi\not\equiv 0\pmod{\mathfrak{q}}$,
hence $\pi-1\equiv 0\pmod{\mathfrak{q}}$, which implies
$2|N(\pi-1)$.

In the other two cases,
complex conjugation is trivial
on the group $(\O/\mathfrak{q})^*$ of odd order.
Note that $\pi\overline{\pi}\in\Q$ implies
that $\pi^2=\pi\overline{\pi}$ is trivial in that group,
hence so is $\pi$. We see again that
$\pi-1\equiv 0\pmod{\mathfrak{q}}$
implies $2|N(\pi-1)$.
\end{proof}

Our second algorithm is a modification of Algorithm~\ref{algo1}
in which we start with an element
$\alpha\in\mathcal{O}_K$, instead of with a prime $p$,
and check if $p=N(\alpha)$ is
a prime that decomposes in the appropriate manner.
We use Algorithm~\ref{algo2} as a stepping stone
towards Algorithm~\ref{algo3},
which allows one to prescribe the embedding degree of the output
by imposing congruence conditions on~$\alpha$.

\noindent\underline{\hspace{\textwidth}}

\begin{algorithm}
\label{algo2}
\hfill\\
\textbf{Input:} A non-Galois {\cmfield } $K$ of degree $4$ and a positive integer~$\ell$.\\
\textbf{Output:} A prime $p$ of $\ell$ bits and a Weil $p^2$-number
corresponding to the Jacobian $\jacC$ of a curve $C$ of genus~$2$ over
$\F_{p^2}$ such that $\jacC$ has
$p$-rank~$1$
and a prime number of $\F_{p^2}$-rational points.

\noindent\underline{\hspace{\textwidth}}
\begin{enumerate}
\item \label{algorithm2step1}Take a random element $\alpha$ of
       $\mathcal{O}_K$ of which
        the norm $N(\alpha)$ has $\ell$ bits.
\item If $p=N(\alpha)$ is prime in~$\Z$, continue. Otherwise, go to
 Step~\ref{algorithm2step1}.
\item If the prime $\beta=p\alpha^{-1}\overline{\alpha}^{-1}$ of $\mathcal{O}_{K_0}$
       remains prime in~$\mathcal{O}_K$, then
       let $\pi=\alpha^2\beta$.
       Otherwise, go to Step~\ref{algorithm2step1}.
\item If $N(u\pi-1)$ is prime for some $u\in \{\pm 1\}$,
       then replace $\pi$ by~$u\pi$.
       Otherwise, go to Step~\ref{algorithm2step1}.
\item \Return $p,\pi$.
\end{enumerate}
\vspace{-\baselineskip}

\noindent\underline{\hspace{\textwidth}}
\end{algorithm}
\begin{theorem}\label{thm:algorithm2}
  If Algorithm~\ref{algo2} terminates, then the output is correct.
  
  Fix the input field $K$ and assume that it satisfies
  the conditions
   of Lemma~\ref{lem:alphadensity}. If $K$ has no prime ideal
   of norm $2$, and no prime above $2$ is ramified in $K/K_0$,
   then
   the heuristic expected runtime of the
  algorithm is polynomial in $\ell$.
\end{theorem}
\begin{proof}
  By Lemma~\ref{lem:alpha}, the output $\pi$ is a Weil $p^2$-number
  satisfying the conditions of Lemma~\ref{lem:mainlemma},
  and the corresponding abelian surface $A$ has $\#A(\F_{p^2})=N(\pi-1)$
  rational points, which
  proves that the output is correct.

  Lemma~\ref{lem:alphadensity} shows that among the elements $\alpha$
  of $\mathcal{O}_K$ of prime norm, at least about
  $1$ in every $4h_K$ has the appropriate factorization,
  so if we treat $N(\alpha)$ and $N(\pi-1)$ as random integers
  as we did in the proof of Theorem~\ref{thm:algorithm1}, then
  we find again that the heuristic expected runtime is polynomial in~$\ell$.
\end{proof}
\begin{remark}
  Actually, the heuristic probability of passing from Step~3 to Step~4
  in Algorithm~\ref{algo2}
  is $1/2$ instead of only $(4h_K)^{-1}$
  as can be seen by applying Chebotarev's density theorem
  to the quadratic extension $LH/H$ from the proof of
  Lemma~\ref{lem:alphadensity}.
\end{remark}

Algorithm~\ref{algo3} constructs $p$-rank-$1$ curves with prescribed embedding
degree by imposing congruence conditions on $\alpha$
in a way that is similar to what is done in the algorithm
of Freeman, Stevenhagen, and Streng~\cite{freeman-stevenhagen-streng}.

\noindent\underline{\hspace{\textwidth}}

\begin{algorithm}\label{algo3}
\hfill\\
\textbf{Input:} A non-Galois {\cmfield } $K$ of degree~$4$, a positive integer $\kappa$
and a prime number $r\equiv 1 \pmod {2\kappa}$ that splits completely in~$K$.\\
\textbf{Output:} A prime $p$ and a Weil $p^2$-number $\pi$ corresponding
to the Jacobian $\jacC$ of a curve $C$ of genus~$2$ over $\F_{p^2}$ that has
$p$-rank~$1$ and embedding degree $\kappa$ with respect to a subgroup of order~$r$.

\noindent\underline{\hspace{\textwidth}}
\begin{enumerate}
\item \label{algorithm3step1}Let $\mathfrak{r}$ be a prime of $K$ dividing $r$,        let
       $\mathfrak{s}=r\mathfrak{r}^{-1}\overline{\mathfrak{r}}^{-1}$
       and compute a basis $b$ of~$\O_K$.
\item \label{algorithm3step2}Take a random element $x$ of $\F_r^*$
       and a primitive $2\kappa$-th root of unity $\zeta\in\F_r^*$.
\item \label{algorithm3step3}Take the `small'
       $\alpha\in\mathcal{O}_K$ such that
       $\alpha\bmod \mathfrak{r}=x$, $\alpha\bmod \overline{\mathfrak{r}}=x\zeta$
       and $\alpha\bmod \mathfrak{s}=x^{-1}$.
       Here `small' means that the coordinates
       with respect to the basis $b$ are $\leq r/2$,
       and $x^{-1}$ is interpreted with respect to the natural
       inclusion of $\F_r^*$ into $\mathcal{O}_K / \mathfrak{s}$. 
\item \label{algorithm3step4}If $p=N_{K/\Q}(\alpha)$ is prime in $\Z$,
continue. Otherwise, go to Step~\ref{algorithm3step2}.
\item \label{algorithm3step5}If
       the prime $\beta=p\alpha^{-1}\overline{\alpha}^{-1}$ of
       $\mathcal{O}_{K_0}$ remains prime in $\mathcal{O}_K$,
       let $\pi=\alpha^2\beta$.
       Otherwise, go to Step~\ref{algorithm3step2}.
\item \Return $p,\pi$.
\end{enumerate}
\vspace{-\baselineskip}

\noindent\underline{\hspace{\textwidth}}
\end{algorithm}

\begin{theorem}\label{thm:algorithm3}
  If Algorithm~\ref{algo3} terminates, then the output is correct.
  If the input field $K$ is fixed and satisfies the conditions of
  Lemma~\ref{lem:alphadensity},
  then the heuristic expected runtime of the
  algorithm is polynomial in $r$.
\end{theorem}
\begin{proof}
The facts that the
output has $p$-rank~$1$ and a Jacobian of order $N(\pi-1)$
are proven as in the proof of Theorem~\ref{thm:algorithm2}.

If $r$ divides the group order $N(\pi-1)$, then the embedding degree
is the order of $(p^2\bmod r)$ in the group $\F_r^*$
(see also \cite[Proposition 2.1]{freeman-stevenhagen-streng}).
So to prove that $\jacC$ has embedding degree $\kappa$
with respect to $r$, it suffices to 
prove that
 $p^2\bmod r$ is a primitive $\kappa$-th root of unity in $\F_r^*$
  and that $r$ divides $N(\pi-1)$.
 
  Let $\phi$ be the non-trivial automorphism of $K_0$.
  Then we have $\beta=\phi(\alpha\overline{\alpha})$,
  hence $\pi\bmod \mathfrak{r}=(\alpha\bmod
    \mathfrak{r})^2(\phi(\alpha\overline{\alpha})\bmod \mathfrak{r})$.
  Inside $\F_r$, we have
  \begin{eqnarray*}
    (\phi(\alpha\overline{\alpha})\bmod \mathfrak{r})
    &=& (\alpha\overline{\alpha}\bmod\mathfrak{s})
        = (\alpha\bmod \mathfrak{s})(\alpha\bmod \overline{\mathfrak{s}})\\
        &=& (\alpha\bmod \mathfrak{s})^2=x^{-2},
  \end{eqnarray*}
  hence we have $(\pi\bmod \mathfrak{r})=1$, so $r$ divides
  $N(\pi-1)$.
  Moreover,
  \begin{eqnarray*}
     (p^2\bmod r)&=&(p^2\bmod \mathfrak{r})
        =(\alpha\bmod\mathfrak{r})^2(\overline{\alpha}\bmod\mathfrak{r})^2
              (\phi(\alpha\overline{\alpha})\bmod\mathfrak{r})^2\\
       &=&(\alpha\bmod\mathfrak{r})^2
       ({\alpha}\bmod\overline{\mathfrak{r}})^2x^{-4}=\zeta^2
   \end{eqnarray*}
    is a primitive $\kappa$-th root of unity.  

  This finishes the proof of the correctness of the output.
  Next we prove the heuristic runtime.
  As $r$ splits completely, 
  $\alpha$ is a lift of some element modulo~$r$.
  We treat its norm $p=N(\alpha)$ as a random
  integer of $4\log_2 r$ bits.
  The rest of the proof is as the proof of Theorem~\ref{thm:algorithm2}.
\end{proof}
\begin{remark}
  Actually, the prime $r$ does not need to split completely
  in Algorithm~\ref{algo3}.
  It suffices to have
  $r\O_K=\mathfrak{r}\overline{\mathfrak{r}}\mathfrak{s}$,
  where $\mathfrak{r}$ is prime and $\mathfrak{s}$ may
  be prime or composite.
\end{remark}
\begin{remark}
  Note that if Algorithm~\ref{algo2} or~\ref{algo3} terminates,
  then $K$ satisfies the conditions of Lemma~\ref{lem:alphadensity},
  which are therefore not only sufficient, but also necessary
  for each of these algorithms to terminate.
\end{remark}

Let 
$A$ be a $g$-dimensional abelian variety over the finite field $k$
of $q$ elements. Its \emph{$\rho$-value}\index{rho-value@$\rho$-value}
with respect to a subgroup of $A(k)$ of order $r$ is defined
to be $\rho=g\log q/\log r$.
As we have $\log \#A(k)\approx g\log q$,
the $\rho$-value measures the ratio between the bit
size of $r$ and the bit size of the order of the
full group of rational
points on~$A$. It is at least about $1$ if $q$ is large.
If we have $A=\jacC$, then a point on $A$ can be represented by
a $g$-tuple of points on $C$, hence $\rho$ is also 
the ratio between the bit size of a group element of $A$
and the bit size of $r$. For cryptography, one
wants the
$\rho$-value to be as small as possible
to save bandwidth when transmitting points on~$\jacC$.

The prime $p$, computed as the norm of the element $\alpha$ in
Step~\ref{algorithm3step4},
is expected to satisfy $\log(p) \approx 4\log(r)$.
Since
our $p$-rank-$1$ curve is defined over~$\F_{p^2}$,
its $\rho$-value is
$\rho = 2\log(p^2)/\log(r) \approx 16$.
For a more detailed version of this heuristic analysis
of the $\rho$-value,
see Freeman, Stevenhagen, and Streng \cite{freeman-stevenhagen-streng},
who compute a $\rho$-value of about $8$ for their
ordinary abelian surfaces with prescribed embedding degree.
For cryptographic applications, a $\rho$-value
of $16$ or even $8$ is larger than desired,
but it does show that pairing-based cryptography
is possible for curves of genus~$2$ with $p$-rank~$1$.

When working with odd embedding degree $\kappa$,
the \emph{embedding field}\index{embedding field} $\F_p(\zeta_r)$ could
be smaller than the field $\F_{p^2}(\zeta_r)=\F_{p^{2\kappa}}$
that is suggested by 
the embedding degree $\kappa$ (see also Hitt
\cite{hitt-2007}).
This may influence the security of pairing-based cryptography,
but can easily be avoided by restricting to even embedding
degree $\kappa$, or by only accepting primes $p$
such that $r$ does not divide $p^\kappa-1$.

\section{Constructing curves with given Weil numbers}\label{sec:cmmethod}

We will now explain the explicit CM construction of a curve $C/\F_{p^2}$
such that $J(\redC)$ corresponds to our Weil $p^2$-number~$\pi$.
A more detailed exposition can be found in~\cite{hehcc18}. 

Honda's\aimention{Honda, Taira}\index{CM method} CM construction
of the abelian variety corresponding to a given Weil $q$-number
$\pi$ is based on the theory of \emph{complex multiplication} of abelian varieties
of Shimura and Taniyama \cite[in particular \S13, Thm.~1]{shimura-taniyama}.
The analogous theory for elliptic curves is even more classical
and dates back to the early 19th century. The first algorithmic application
of the CM construction of elliptic curves
is its application to primality proving by Atkin and Morain \cite{atkin-morain}.

The construction starts by taking an abelian variety $\liftA$ over a number field $F$
such that we have $\mathrm{End}(\liftA)\cong \O_K$, where $K$ is a field
containing $\pi$,
and reduces this variety modulo an appropriate prime $\mathfrak{P}$ of~$F$.
For our $p$-rank-$1$ Weil numbers $\pi$, one can take $K=\Q(\pi)$ and any prime
$\mathfrak{P}$ dividing~$p$.

In the dimension-$2$ case, instead of writing down the abelian
surface $\liftA$ itself, one only writes down
the \emph{absolute Igusa invariants}\index{Igusa!invariants}\index{absolute Igusa invariants} $j_1,j_2,j_3\in F$
of the curve $\liftC$ of which $\liftA$
is the Jacobian.
These invariants are the first three of a set of $10$ invariants
given on page 641 of \cite{igusa}.
One then reduces the invariants modulo $\mathfrak{P}$
and, assuming $(j_1\bmod\mathfrak{P})$ is a unit,
constructs $\redC=(\liftC\bmod\mathfrak{P})$ from the reduced invariants
using Mestre's algorithm~\cite{mestre}.
Honda's construction shows that $J(\redC)$ or its quadratic twist
corresponds to our Weil $p^2$-number~$\pi$.

In all practical implementations, the invariants $j_n\in F$ are represented by
polynomials $H_1,H_2,H_3$ or $H_1,\widehat{H}_2,\widehat{H}_3$
called \emph{Igusa class polynomials}.\index{Igusa!class polynomials}\index{class polynomial!Igusa}
We explain the polynomials $\widehat{H}_n$ later, but the polynomials $H_n$ are given by
$$H_n=\prod_{\liftC}(X-j_n(\liftC)),$$
where the product ranges over isomorphism classes of curves $\liftC$ such that
we have $\mathrm{End}(J(\liftC))\cong \O_K$.
For every triple $(j_1,j_2,j_3)$ of
zeroes $j_n\in\overline{\F_p}$ of $H_n$ with $j_1\not=0$,
one thus obtains
a unique $\overline{\F_p}$-isomorphism class of curves.
Assuming $j_1(\liftC)\not\in\mathfrak{P}$ for some~$\liftC$,
a twist of at least one of the curves we obtain has Weil number $\pi$.
Let $\redC$ be such a curve.
As we know the group order $N(\pi-1)$
of $J(\redC)(\F_{p^2})$, we can
quickly check whether we have the correct curve by taking random points
on its Jacobian and multiplying them by $N(\pi-1)$.

As the field $K$ is fixed, so are its class polynomials.
They can therefore be precomputed using any of the three known algorithms:
the complex analytic method of
Spallek \cite{spallek} and van Wamelen \cite{vanwamelen}, for which Streng \cite{streng} recently gave the first runtime analysis and proof of correctness,
the $2$-adic method of Gaudry, Houtmann, Kohel, Ritzenthaler, and Weng \cite{ghkrw-2adic},
and the Chinese remainder method of Eisentr\"ager and Lauter \cite{eisentrager-lauter}.
Alternatively, class polynomials
can be found in the ECHIDNA database \cite{echidna}.

The alternative class polynomials $\widehat{H}_n$ are given by
$$ \widehat{H}_n=\sum_{\liftC} j_n(\liftC)
\prod_{\liftC'\not\cong \liftC} (X-j_1(\liftC')),\quad (n=2,3)$$
where both the product and the sum range over 
isomorphism classes of curves $\liftC$ for which $\mathrm{End}(J(\liftC))\cong \O_K$ holds.
For any such $\liftC$, we have $j_n(\liftC) H_1'(j_1(\liftC))=\widehat{H}_n(j_1(\liftC))$.
This implies that if every coefficient of $H_1$ has a denominator that is not divisible by~$p$,
and $(H_1\bmod p)$ has a non-zero root of multiplicity~$1$,
then we can compute the
Igusa invariants of a curve $\redC$, which is automatically either
the curve we want or a quadratic twist.
The idea of using $\widehat{H}_n$ and not the more standard Lagrange interpolation
is due to Gaudry, Houtmann, Kohel, Ritzenthaler, and Weng,
who show in \cite{ghkrw-2adic} that $\widehat{H}_n$ heuristically has a much smaller height.

\section{A sufficient and necessary condition for Algorithm 1}
\label{sec:suffnec}

As said before, the condition of Lemma~\ref{lem:alphadensity}
are sufficient for all three algorithms to work and necessary
for Algorithms 2 and~3. They are also easy to check and
true for `most' non-Galois quartic \cmfields. The current section
gives a weaker condition that is both
sufficient and necessary for Algorithm~1 to work. We also give examples
to show that this condition is non-trivial
and strictly weaker than that of Lemma~\ref{lem:alphadensity}.

Let $K$ be a non-Galois {\cmfield } of degree $4$.
Let $\liftC/\overline{K}$ be
a curve of genus~$2$ over the algebraic closure $\overline{K}$ of $K$
such that $\mathrm{End}(J_{\liftC})\cong \O_K$ holds.
Such $C$ are known to exist.
The field $\Q(j)\subset \overline{K}$ generated
over $\Q$ by all $10$ absolute Igusa invariants $j_1(\liftC),\ldots,j_{10}(\liftC)$
of \cite[page 641]{igusa}
is called the \emph{field of moduli}\index{field of moduli}
of~$\liftC$. For any subfield $X\subset \overline{K}$,
let $X(j)$ be the compositum $X\cdot \Q(j)$.
Write $K=K_0(\sqrt{r})$ for some $r\in K_0$ and let $K_0^{\mathrm{r}}=\Q(\sqrt{N_{K_0/\Q}(r)})$
(as before).
\begin{lemma}\label{lem:pidensity}
Let $K, K_0^{\mathrm{r}}, K(j)$ be as above and let $G$
be the Galois group of the normal closure of $K(j)$ over~$\Q$.
Let $S$ be the set of primes $p$ that factor in $K$
as $p\O_K=\pfactor$ and such that there exists a Weil
$p^2$-number $\pi$ such that we have $\pi\O_K=\pifactor$.

The Dirichlet density of $S$ is
$$\frac{\#\{\sigma\in G \mid \ord\sigma=2, \sigma_{|K_0^{\mathrm{r}}}\not=\mathrm{id}_{K_0^{\mathrm{r}}}\}}{\# G}.$$
If $S$ is non-empty, then it has positive density.
\end{lemma}
\begin{corollary}\label{cor:pidensity}
  If Algorithm~1
terminates on input~$K$, then $\sigma$ as in
 Lemma~\ref{lem:pidensity} exists for~$K$.
Conversely, if $K$ is fixed and $\sigma$ exists for~$K$,
then Algorithm~1 heuristically has a polynomial runtime.
\end{corollary}
\begin{proof}[{\textbf{\textup{Proof of Corollary \ref{cor:pidensity}}}}]
  If Algorithm~1 terminates, then $S$ is non-empty, hence
  $\sigma$ exists by Lemma~\ref{lem:pidensity}.
  If $\sigma$ exists, then the proof of Theorem~\ref{thm:algorithm1}
  is valid, so Algorithm~1 heuristically has a polynomial runtime.
\end{proof}
To prove Lemma~\ref{lem:pidensity}, we need some more theory.
Let $L$ be the normal closure of $K$.
A \emph{\cmtype}\index{CM-type} of $K$ is a set $\Phi$
of two embeddings $\varphi:K\rightarrow L$ that satisfies
$\Phi\cap\overline{\Phi}=\emptyset$.
Let $\liftC$ be a curve as above,
and let $\Phi=\{\varphi_1,\varphi_2\}$ be its {\cmtype } as defined in
\cite[\S5.2]{shimura-taniyama}.
The exact definition of this {\cmtype } will not be important to us.

The \emph{reflex field}\index{reflex!field}
$$K^{\mathrm{r}}=\Q(\textstyle{\sum_i\varphi_i(x)} : x\in K)\subset L$$
of $K$ with respect to $\Phi$
is one of the two non-Galois CM subfields of $L$ of degree $4$ that are not
conjugates of~$K$.
Its real quadratic subfield $K_0^{\mathrm{r}}$
does not depend on $\Phi$ and is exactly the field $K_0^{\mathrm{r}}$
that we have seen above Lemma \ref{lem:alphadensity}.
By \cite[Prop.~20.3(i)]{shimura}, we have $K_0^{\mathrm{r}}\subset \Q(j)$,
so that we have the inclusions of fields shown in Figure~\ref{figure:fielddiagram}.
\begin{figure} 
\[
\begin{xy}
    \xymatrix{
    & & & K^{\mathrm{r}}(j) \\
    & L & & \\
    K \ar@{-}[ru]^2 & & K^{\mathrm{r}} \ar@{-}[lu]_2 \ar@{-}[ruu]&
        \Q(j)\ar@{-}[uu]\\
    &&&\\
    K_0 \ar@{-}[uu]^2 & & K_0^{\mathrm{r}} \ar@{-}[uu]_2 \ar@{-}[ruu]& \\
    & \Q \ar@{-}[lu]^2 \ar@{-}[ru]_2& & 
    }
\end{xy}
\]
\caption{Inclusions between the fields
\label{figure:fielddiagram}}
\end{figure}
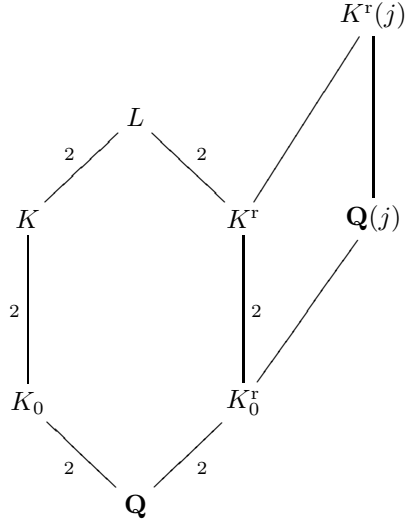

The main theorem of complex multiplication gives $K^{\mathrm{r}}(j)$ as
an unramified abelian extension of~$K^{\mathrm{r}}$.
To state it, we need to define the \emph{type norm} of the \emph{reflex type} of~$\Phi$.
Let $\Phi_L$ be the set of extensions of elements of $\Phi$ to $L$,
so $\Phi_L$ is a {\cmtype } of $L$ and so is the set
$\Phi_L^{-1}$ of inverses of elements of $L$.
The set of restrictions of $\Phi_L^{-1}$ to $K^{\mathrm{r}}$
is a {\cmtype } $\Phi^{\mathrm{r}}=\{\psi_1,\psi_2\}$
of $K^{\mathrm{r}}$ called the \emph{reflex}\index{reflex!CM-type} of $\Phi$
\cite[\S8.3]{shimura-taniyama}.
By \cite[\S8.3 Prop.~29]{shimura-taniyama},
for any fractional $\O_{K^{\mathrm{r}}}$-ideal~$\mathfrak{a}$,
there is a unique fractional
$\O_K$-ideal $N_{\Phi^{\mathrm{r}}}(\mathfrak{a})$
such that we have
$$N_{\Phi^{\mathrm{r}}}(\mathfrak{a})\O_L=\prod_{i=1}^2\psi_i(\mathfrak{a})\O_L.$$
The map $N_{\Phi^{\mathrm{r}}}$ from ideals of $K^{\mathrm{r}}$ to ideals
of $K$ is called the \emph{type norm}\index{type norm} with respect to $\Phi^{\mathrm{r}}$.
\begin{theorem}[{Main Theorem 1 in \S15.3 of \cite{shimura-taniyama}}]\label{thm:mainthmcm}
   The extension $K^{\mathrm{r}}(j)/K^{\mathrm{r}}$
   is abelian and unramified.
   Its Galois group corresponds via the Artin
   map to $\cl_{K^{\mathrm{r}}}/H_0$,
   where $H_0$ is the group of ideal classes $[\mathfrak{a}]$
   such that $N_{\Phi^{\mathrm{r}}}(\mathfrak{a})$
   is principal and generated by
   an element $\mu\in K$ with $\mu\overline{\mu}\in \Q^*$.\qed
\end{theorem}
The following lemma computes $N_{\Phi^{\mathrm{r}}}(\mathfrak{q})$ for certain primes~$\mathfrak{q}$.
\begin{lemma}\label{lem:computetypenorm}\label{lem:inreflex}
  Let $K$ be a quartic {\cmfield } and $p$ a prime that
  factors in $K$ as $p\O_K=\pfactor ^e$.
\begin{enumerate}
 \item The prime $p$ factors in $K_0^{\mathrm{r}}$ as $\mathfrak{s}^e$ for a prime $\mathfrak{s}$,
       which  splits in $K^{\mathrm{r}}$
       as $\mathfrak{s}\O_{K^{\mathrm{r}}}=\mathfrak{q}\overline{\mathfrak{q}}$; and
 \item we have $N_{\Phi^{\mathrm{r}}}(\mathfrak{q})=\mathfrak{p}_1^{\smash{2/e}}\mathfrak{p}_2$
       (up to complex conjugation).
\end{enumerate}
\end{lemma}
\begin{proof}
   Let $\mathfrak{P}\subset\O_L$ be the unique prime over~$\mathfrak{p}_1$.
   Part (1) follows from the fact that the decomposition group
   of $\mathfrak{P}$
   is $\Gal(L/K)$ and that the inertia group has order~$e$.
   
   For part (2), let $s$ be the generator of $\Gal(L/K)$, let
   $s'$ be the generator of $\Gal(L/K^{\mathrm{r}})$
   and set $r=ss'$.
   Then $\Phi_L\subset\Gal(L/\Q)$ has $4$ elements and satisfies
   $\Phi_L\langle s\rangle = \Phi_L$ and
   $\Phi_L^{-1}\langle s'\rangle=\Phi_L^{-1}$, hence
   $\Phi_L^{-1}$ is $\{1,s,s',ss'\}$ or its complex conjugate,
   and we have $\Phi^{\mathrm{r}}=\{1,s_{|K^{\mathrm{r}}}\}$
   up to complex   conjugation.
   Take $\psi_1=1,\psi_2=s$.
   We compute
   \begin{eqnarray*}N_{\Phi^{\mathrm{r}}}(\mathfrak{q})\O_L&=&
     (\mathfrak{q}\O_L)({}^s\mathfrak{q}\O_L)=
     \left(\mathfrak{P} ({}^{s'}\mathfrak{P})\right)\left(({}^s\mathfrak{P})( {}^{ss'}\mathfrak{P})\right)\\
     &=&\mathfrak{P}^2 \left(({}^{s'}\mathfrak{P})( {}^{ss'}\mathfrak{P})\right)
     =(\mathfrak{p}_1^{2/e}\O_L) (\mathfrak{p}_2\O_L),
\end{eqnarray*}
   up to complex conjugation, which proves~(2).
\end{proof}
\begin{proof}[\textbf{\textup{{Proof of Lemma~\ref{lem:pidensity}}}}]
  Let $p$ be a prime number that is unramified in~$K$.
  We prove that $p$ is in $S$ if and only if its
  decomposition group in the normal closure of $K(j)$
  is of order $2$ and acts non-trivially on $K_0^{\mathrm{r}}$.
  Chebotarev's density theorem~\cite[Theorem 13.4]{neukirch} then 
  proves the formula for the density. Moreover, if $S$
  is non-empty, then $\sigma$ exists, hence
  the density is positive.
  
  Let $p$ be a prime number and let $\sigma\in G$ be its
  $p$-th power Frobenius.
  Suppose $p$ is in $S$ and write $p\O_K=\pfactor $.
  The image of $\sigma$ in $\Gal(L/\Q)$
  generates $\Gal(L/K)$ or its conjugate, hence has order~$2$.
  It follows that $p$ is inert in $K_0^{\mathrm{r}}/\Q$
  and splits into two factors $\mathfrak{q}$ and
  $\overline{\mathfrak{q}}$ in $K^{\mathrm{r}}$.
  Lemma \ref{lem:computetypenorm} shows that the type norm of $\mathfrak{q}$ is
  $N_{\Phi^{\mathrm{r}}}(\mathfrak{q})=\pifactor =\pi\O_K$
  or its complex conjugate, and we have $\pi\overline{\pi}\in\Q^*$,
  so we find $[\mathfrak{q}]\in H_0$, hence $\sigma^2$ is trivial
     on~$K^{\mathrm{r}}(j)$ and in particular on $\Q(j)$.

  Recall that $\Q(j)$ is the field generated over $\Q$ by
  the absolute Igusa invariants of $\liftC$ and that $\liftC$ is any
  curve with CM by $\mathcal{O}_K$.
  In particular, we can replace $\liftC$ by ${}^\tau \liftC $ for any
  automorphism $\tau$ of~$\overline{K}/\Q$.
  This shows that $\sigma^2$ is also trivial on ${}^\tau\Q(j)$
  for any $\tau$, and hence $\sigma^2$ is trivial on
  the normal closure of $\Q(j)$.
  As it is also trivial on the normal closure $L$ of $K$,
  we find that it is trivial on the normal closure of $K(j)$
  and hence $\sigma$ is in the set of Lemma~\ref{lem:pidensity}.

  Conversely, suppose that $\sigma^2$ is trivial
  and $\sigma$ is non-trivial on $K_0^{\mathrm{r}}$.
  As $\sigma_{|L}$ generates $\Gal(L/K)$ or a conjugate,
  we find that $p$ factors
  as $p\O_K=\pfactor $.
  Again, the prime $p$ is inert in $K_0^{\mathrm{r}}/\Q$
  and splits into two factors $\mathfrak{q}$ and
  $\overline{\mathfrak{q}}$ in $K^{\mathrm{r}}$
  with type norms $\pifactor $ and its complex conjugate.
  As we have $\sigma^2=1$, we find by Theorem \ref{thm:mainthmcm} that $\pifactor =\pi\O_K$ holds
  for some $\pi\in\O_K$ that satisfies $\pi\overline{\pi}\in \Q^*$. Since also 
  $\pi\overline{\pi}$ is positive and has absolute value $p^2$,
  it is a Weil $p^2$-number
  and $p$ is in $S$.
\end{proof}
\begin{example}
  For the field $K=\Q[X]/(X^4+12X^2+2)$
  of Example~\ref{ex:firstexample}, we can find
  $\Q(j)$ in the ECHIDNA database \cite{echidna}
  and compute that $\Q(j)$
  contains the field $F=\Q(\sqrt{2+\sqrt{2}})$, which is cyclic
  Galois over $\Q$ and contains $K_0^{\mathrm{r}}=\Q(\sqrt{2})$.
  Any automorphism of $F$ of order $2$ is trivial on $K_0^{\mathrm{r}}$,
  so the density of $S$ in Lemma~\ref{lem:pidensity} is~$0$
  and none of our algorithms works for this field.
\end{example}
\begin{example}
  For the field $K=\Q[X]/(X^4+20X^2+5)$, we have $13\in S$,
  so that $S$ has positive density and Algorithm~1 works for~$K$.
  However, the discriminant $d^{\mathrm{r}}=5$ of
  $K_0^{\mathrm{r}}=\Q(\sqrt{5})$ is a prime discriminant
  and occurs in the prime discriminant factorization
  $d=(-4)\cdot(5)\cdot( -19)$
  of $K_0$.
  This shows that $K$ does not satisfy the conditions of Lemma \ref{lem:alphadensity}, which are therefore too strong for Algorithm~1.
\end{example}

\section{Factorization of class polynomials modulo \texorpdfstring{$p$}{p}}
\label{sec:factorclasspolynomials}

While experimenting with the explicit CM construction for
curves of $p$-rank~$1$, we found that in the (ramified) case $e=2$
of Lemma \ref{lem:mainlemma}, the polynomial $H_1\bmod p$ has
no roots of multiplicity~$1$ in $\overline{\F_p}$,
which made working with $\widehat{H}_n$
impossible.
The current section explains this phenomenon, and shows how to adapt
$H_1,\widehat{H}_2,\widehat{H}_3$ to deal with this situation.
We also explain the analogue of this for the situation~$e=1$,
for which there is no problem.

Let $K$, $\liftC$, and $j$ be as in Section~\ref{sec:suffnec}.
If $j_1(\liftC)\not=0$ is a simple root of~$H_1$, which is `usually' the case,
then we have $\Q(j)=\Q(j_1(\liftC))$
since we can compute $j_n(\liftC)$ from $j_1(\liftC)$ using the polynomials $\widehat{H}_2$
and $\widehat{H}_3$
as we have seen in Section~\ref{sec:cmmethod}.
The Kummer-Dedekind theorem thus relates the factorization
of $(H_1\bmod p)\in\F_p[X]$ to the factorization of $p$ in (an order in)~$\Q(j)$.

\begin{lemma}\label{lem:incmclassfield}
  Let $p$ be a prime that factors in $K$
  as $p\O_K=\pfactor$,
  and let $n$ be the smallest positive integer such that
  $en$ is even and
  $(\mathfrak{p}_1\mathfrak{p}_2^{\smash{e/2}})^{n}$
  is generated
  by a Weil $p^n$-number~$\pi$.
  Then any prime $\mathfrak{q}$ of $K^{\mathrm{r}}$ lying over $p$
  decomposes in $K^{\mathrm{r}}(j)/K^{\mathrm{r}}$ into
  distinct primes of residue degree $en/2$.
\end{lemma}
\begin{proof}
Recall from Theorem~\ref{thm:mainthmcm} that $K^{\mathrm{r}}(j)$
is the unramified abelian extension of $K^{\mathrm{r}}$
such that the Artin map induces an isomorphism
$\cl_K/H_0\rightarrow \mathrm{Gal}(K^{\mathrm{r}}(j)/K^{\mathrm{r}})$,
where $H_0\subset\cl_K$ 
is the subgroup
of ideal classes $[\mathfrak{a}]$
such that $N_{\Phi^{\mathrm{r}}}(\mathfrak{a})$
is principal and generated by
an element $\mu\in K$ with $\mu\overline{\mu}\in \Q^*$.

The Artin isomorphism 
sends $[\mathfrak{q}]$ to a generator of the decomposition
group of~$\mathfrak{q}$, so it suffices to prove that $[\mathfrak{q}]$
has order $en/2$ in the quotient group $\cl_{K^{\mathrm{r}}}/H_0$.
Lemma \ref{lem:computetypenorm} computes that  $N_{\Phi^{\mathrm{r}}}(\mathfrak{q}^m)$
is either $(\mathfrak{p}_1^{\smash{2/e}}\mathfrak{p}_2)^m$ or
its complex conjugate,
so the smallest integer $m$ with $[\mathfrak{q}^m]\in H_0$
is exactly $m=en/2$.
\end{proof}
\begin{corollary}\label{cor:splittinginqj}
  Let $p,n$ be as in Lemma \ref{lem:incmclassfield}. Then $p$ splits
  into prime factors of residue degree $n$ in $\Q(j)/\Q$.
  Each factor occurs exactly $e$ times.
\end{corollary}
\begin{proof}
  Each prime factor $\mathfrak{p}$ has residue degree $en/2$ in $K^{\mathrm{r}}(j)/K^{\mathrm{r}}$
  by Lemma \ref{lem:incmclassfield} and $2/e$ in $K^{\mathrm{r}}/\Q$ by Lemma \ref{lem:inreflex},
  hence $n$ in $K^{\mathrm{r}}(j)/\Q$.
  As all ramification of $\mathfrak{p}$ takes place in $K^{\mathrm{r}}_0/\Q$,
  we find that the ramification index of $\mathfrak{p}$
  in $K^{\mathrm{r}}(j)/\Q$ is~$e$.
  
  We have seen in Figure~\ref{figure:fielddiagram}
  on page~\pageref{figure:fielddiagram} that $\Q(j)$
  contains $K_0^{\mathrm{r}}$.
  As the residue degree and ramification index of $\mathfrak{p}$
  in $K^{\mathrm{r}}/K_0^{\mathrm{r}}$ are $1$,
  we find that
  the residue degree and ramification
  index of $\mathfrak{p}$ are also $n$ and $e$ in $\Q(j)/\Q$.  
\end{proof}
\begin{corollary}\label{cor:explainphenomenon}
  If $p$ factors in $K$ as $p\O_K=\pfactor ^2$,
  then $(H_1\bmod p)\in\F_p[X]$ has no roots of multiplicity~$1$ in $\overline{\F_p}$.
\end{corollary}
\begin{proof}
  The polynomial $H_1\in\Q[X]$ is monic and the denominators
  of the coefficients are not divisible by $p$ because they are Igusa invariants
  of a curve that has potential good reduction modulo~$p$.
  Let $c\in\Z$ not divisible by $p$ be such that $H_1(cX)$ is in $\Z[X]$ and
  let $f\in\Z[X]$ be an arbitrary irreducible factor of $H_1(cX)\in\Z[X]$.
  We find an order $\O=\Z[X]/f$ in~$\Q(j)$.
  Each irreducible factor $g\in\F_p[X]$ of $(H_1\bmod p)$ corresponds to the prime
  ideal $\mathfrak{p}=(p,g(X))$ of $\O$.
  As every prime over $p$ ramifies in $\Q(j)/\Q$ by Corollary~\ref{cor:splittinginqj}, we find that $\mathfrak{p}$
  is either ramified or singular. 
  By the Kummer-Dedekind theorem (Theorem~8.2 of \cite{stevenhagen}),
  both cases imply that the roots of $g$ have multiplicity at least $2$
  as roots of~$H_1$.
\end{proof}
This shows that $H_1,\widehat{H}_1,\widehat{H}_2$ cannot be used
for the case $e=2$.
To get around this, we replace $H_1$ by an irreducible
factor $f\in K_0^{\mathrm{r}}[X]$ and
$\widehat{H}_n$ by the unique polynomial $S_{n}$ of degree at most $\mathrm{deg}(f)-1$
that is congruent modulo $f$ to $\widehat{H}_{n}(H_1/f)^{-1}$.
If we write $p\O_{K^{\mathrm{r}}}=\mathfrak{s}^2$, 
then $(f\bmod\mathfrak{s}),(S_2\bmod\mathfrak{s}),(S_3\bmod\mathfrak{s})\in \F_p[X]$
can be used in exactly the same way as
$(H_1\bmod p),(\widehat{H}_2\bmod p),(\widehat{H}_3\bmod p)$
and do not suffer from Corollary~\ref{cor:explainphenomenon}.
\begin{corollary}
  For all but finitely many of the primes $p$ that decompose
    as $p\O_K=\pfactor ^e$,
    the reduction $(H_1\bmod p)\in\F_p[X]$ is a product of distinct irreducible
    polynomials in $\F_p[X]$ of degree $n$ for $n$
    given in Lemma~\ref{lem:incmclassfield} (and depending on $p$).
\end{corollary}
\begin{proof}
  We exclude the primes dividing the denominator of any coefficient of~$H_1$,
  as well as those dividing the discriminant.
  Then all roots of $(H_1\bmod p)$ in $\overline{\F_p}$ are simple roots.
  Let $f,\O$ be as in the proof of Corollary \ref{cor:explainphenomenon}.
  Then $p$ does not divide the index of $\O$ in its maximal order.
  The fact that every prime of $\Q(j)$ has residue degree $n$ implies
  that every irreducible factor of $f\bmod p$ has degree~$n$.
\end{proof}

\section{Examples}\label{sec:exalg1}\label{examples}
\subsection*{Algorithm 1}

We provide examples of $p$-rank-$1$ curves $C/\F_{p^2}$
such that the Jacobian $J_C$ is simple and has prime order.
The {\cmfield } for all examples is
$K=\Q(\alpha)$, where $\alpha$ is a root of
the polynomial $X^4 + 34X^2 + 217 \in \Q[X]$,
which satisfies the conditions of Lemma~\ref{lem:alphadensity}.
We give the prime~$p$, the coefficients
$a_1$ and $a_2$ of the minimal polynomial
$$f=X^4-a_1X^3+(a_2+2p^2)X^2-a_1p^2X+p^4$$
of the Frobenius endomorphism and the
coefficients $c_i \in \F_{p^2}$ of the curve equation
\[C: y^2 = c_6x^6 + c_4x^4 + c_3x^3 + c_2x^2 + c_1x + c_0.\]
The group order of the Jacobian is 
$\#J_C(\F_{p^2}) = N(\pi-1)=f(1)$. 
The field $\F_{p^2}$ is given as $\F_p(\sigma)$, where $\sigma^2=-3$.
Section
headings describe the number of bits of
the group order $\#J_C(\F_{p^2})$.

Each example was generated in a few seconds on a standard PC after pre-computation
of the Igusa class polynomials of~$K$.
\subsubsection*{160-bit group size}
\begin{eqnarray*}
p & = & 924575392409, \quad
a_1  =  3396725192754\\
a_2 & = &2876182159630959921399337,\quad
 c_6  = \sigma \\
 c_4 & = & 349419850452\cdot\sigma + 621473390194\\
 c_3 & = & 638315825844\cdot\sigma + 895470286740\\
 c_2 & = & 247903071476\cdot\sigma + 504258872407\\
 c_1 & = & 494346973570\cdot\sigma + 326558224146\\
 c_0 & = & 721392332677\cdot\sigma + 210623692149
\end{eqnarray*}

\subsubsection*{192-bit group size}
\begin{eqnarray*}
p & = & 236691298903769,\quad
a_1  =  -9692493559086\\
a_2 & = &-58992172275797931791883572663,\quad
c_6  =  \sigma\\
c_4 & = & 144046547562595\sigma + 31854049506043\\
c_3 & = & 134634542821316\sigma + 20155601614364\\
c_2 & = & 159093189820788\sigma + 52669766944798\\
c_1 & = & 223684436822489\sigma + 66232364455191\\
c_0 & = & 206430094481010\sigma + 170879851904277
\end{eqnarray*}

\subsubsection*{256-bit group size}
\begin{eqnarray*}
p & = & 15511800964685067143,\quad
a_1  =  2183138494024250742\\
a_2 & = & -871403391229975003782565554464700664457,\quad
c_6 = 1\\
c_4 &=& 7019198877313644539\cdot\sigma + 8886572032497699458\\
c_3 &=& 8069566800142565548\cdot\sigma + 11092851174307405252\\
c_2 &=&8339873208295381793\cdot\sigma + 13688811293938352344\\
c_1 &=&10474983032301001361\cdot\sigma + 14509908493781086362\\
c_0 &=&4803877905347330504\cdot\sigma+ 12900291622358663970
\end{eqnarray*}


\subsection*{Algorithm 3}

\subsubsection*{192-bit group size, embedding degree $12$}

Let $K$ be the field $K=\Q[X]/(X^4+13X^2+41)$ and
let $\kappa=12$. It took a few seconds
to find the smallest prime $r>2^{192}$ that
splits completely in $K$ and
$\Q(\zeta_{12})$, which is $r=2^{192}+18513$.
We ran Algorithm~3 with input $K,\kappa,r$.
The algorithm terminated after about $11$ minutes and
found a prime $p$ and a Weil $p^2$ number with $p$-rank~$1$
and embedding degree $12$ with respect to a subgroup of order~$r$.
Using $p$ and precomputed Igusa class polynomials, we were able
to find an equation for the corresponding hyperelliptic
curve $C$ in less than a second.
We only give $p$, because $\pi$ and the coefficients of $C$
would take up too much space.
\small
\begin{eqnarray*}p
&=&
1420038565958074827476353870489770880715201360323415690146120568\\
& &6404970976014364663695672498066437749119607973051961772352102985\\
& &5649462172148699393958968638652107696147277436345811056227385195\\
& &781997362304851932650270514293705125991379
\end{eqnarray*}
\normalsize



\bigskip

\noindent {\bf Acknowledgements.}
We thank
Peter Bruin,\aimention{Bruin, Peter}
David Kohel,\aimention{Kohel, David}
Tanja Lange,\aimention{Lange, Tanja}
Hendrik~Lenstra,\aimention{Lenstra, Hendrik W.}
Joe Silverman,\aimention{Silverman, Joseph H.}
and Peter Stevenhagen\aimention{Stevenhagen, Peter}
for helpful advice.

\bibliographystyle{plain}
\bibliography{CMbib}

\end{document}